\documentclass[11pt]{article}
\usepackage{mathrsfs}
\usepackage{amsfonts}
\usepackage[centertags]{amsmath}
\usepackage{amssymb}
\usepackage{amsthm}
\usepackage{enumerate}
\usepackage{graphicx}
\usepackage{appendix}
\usepackage{color,xcolor}
\usepackage{epstopdf}
\usepackage{hyperref}

\newtheorem{thm}{Theorem}[section]
\newtheorem{prop}[thm]{Proposition}
\newtheorem{lem}[thm]{Lemma}
\newtheorem{cor}[thm]{Corollary}

\newtheorem{rem}[thm]{Remark}
\newtheorem{exam}[thm]{Example}

 \def\O{\Omega}
\def\og{\overline{G}}

\def\di{\bigm|} \def\lg{\langle} \def\rg{\rangle}
\def\nd{\mathrel{\bigm|\kern-.7em/}}

\def\f{\noindent}

\def\mod{\hbox{\rm mod }}

\def\qed{\hfill $\Box$}

\begin{document}
\title{Finite ${\mathcal DC}$-groups
\thanks{This research was partially supported by the National Natural Science Foundation of China (nos.11771191, 11771258)}}
\author{ Dandan Zhang$^{1}$,
 Haipeng Qu$^2$ and Yanfeng Luo$^1$\thanks{corresponding author}\\
\footnotesize\em $1$. School of Mathematics and Statistics, Lanzhou University,\\
\footnotesize\em Lanzhou 730000, P. R. China\\
\footnotesize\em e-mail: luoyf@lzu.edu.cn \\
\footnotesize\em $2.$ School of Mathematics and Computer Science, Shanxi Normal University,\\
\footnotesize\em Linfen Shanxi 041004, P. R. China}
\date{}
\maketitle

\begin{abstract}
Let $G$ be a group and ${\rm DS}(G)=\{H'\ |\ H\leq G\}$. $G$ is said to be a ${\mathcal DC}$-group if ${\rm DS}(G)$ is a chain. In this paper, we prove that a finite ${\mathcal DC}$-group is a semidirect product of a Sylow $p$-subgroup and an abelian $p'$-subgroup. For the case of $G$ being a finite $p$-group, we obtain some properties of a ${\mathcal DC}$-group. In particular, a ${\mathcal DC}$ $2$-group is characterized.
Moreover, we prove that ${\mathcal DC}$-groups are metabelian for $p\leqslant 3$ and give an example that a non-abelian ${\mathcal DC}$-group is not be necessarily metabelian for $p\geqslant 5$.


\medskip
\noindent{\it 2010 Mathematics subject classification$:$} 20D15, 20D30, 20F05, 20F14.
\end{abstract}

\baselineskip=16pt

\section{Introduction}

For a group $G$, the set ${\rm L}(G)$ of all subgroups of $G$ is a lattice with respect to set inclusion. Let ${\rm DS}(G)=\{H'\ |\ H\leq G\}$.
Then ${\rm DS}(G)$ is a subset of $ {\rm L}(G)$. However, in general, it is not a sublattice of ${\rm L}(G)$. Example is as follows.

\begin{exam}
Let $G$ be a symmetric group $S_{6}$. Then~~${\rm DS}(G)$ is not a sublattice of~~${\rm L}(G)$.
\end{exam}

\begin{proof}
Let $N=\lg (123),(456)\rg$, $H=N_G(N)$ and $L=\lg (1425)(36), (12)\rg$.
Then $H=N\rtimes L$. It is easy to see that $H'=N\rtimes L'=N\rtimes \lg (12)(45)\rg$. Let $K=A_5$. Then $K'=K$. Clearly $H'\cap K'=\lg (123),(12)(45)\rg\cong S_3$.
If there exists a subgroup $S$ satisfying $S'=H'\cap K'$, then $S\le N_G(S')=N_G(H'\cap K')=\lg (123),(12),(45)\rg\cong S_3\times {\rm C}_2$.
Hence $|S'|\leqslant |(S_3\times {\rm C}_2)'|=3$. This contradicts $|S'|=6$.
Thus ${\rm DS}(G)$ is not a sublattice.
\end{proof}

Naturally, we focus on  the case that ${\rm DS}(G)$ is a sublattice. For the reason that chain is the simplest lattice, we turn our attention to the study of groups $G$ satisfying ${\rm DS}(G)$ is a chain. Obviously, ${\rm DS}(G)$ is a chain is equivalent to $H'\leq K'$ or $K'\leq H'$ for any $H, K\leq G$. For convenience, we call such groups ${\mathcal DC}$-groups and use $G\in {\mathcal DC}$ to denote $G$ is a ${\mathcal DC}$-group.
It is not difficult to prove that a finite ${\mathcal DC}$-group $G$ is a semidirect product of a Sylow $p$-subgroup and an abelian $p'$-subgroup of $G$.
Hence the study of finite ${\mathcal DC}$-groups can be reduced to that of finite ${\mathcal DC}$ $p$-groups, where a finite ${\mathcal DC}$ $p$-group means it is a ${\mathcal DC}$-group and a finite $p$-group. The main results in this paper are as follows.

\begin{thm}\label{thm1.2}
Assume $G$ is a finite ${\mathcal DC}$ $p$-group. Then $d(G')\leqslant p$. In particular, if $d(G')=p$, then $G'\cong {\rm C}^{p}_{p}$.
\end{thm}

\begin{thm}\label{DC2-group}
Assume $G$ is a finite $2$-group. Then $G$ is a ${\mathcal DC}$-group if and only if $G'$ is cyclic or $G'\cong {\rm C}^{2}_{2}$ and ${\rm cl}(G)=3$.
\end{thm}

Obviously, a finite ${\mathcal DC}$ $2$-group is metabelian. We also prove that a finite ${\mathcal DC}$ $3$-group is metabelian. However,
there exists finite non-abelian ${\mathcal DC}$ $p$-groups are non-metabelian for $p\geqslant 5$. The counterexample is given.

\medskip
Finally, it must be mentioned that the set $\{H'\ |\ H\leq G\}$ was introduced by de Giovanni and Robinson in \cite{de},
independently, by Herzog, Longobardi and Maj in \cite{Herzog}. They investigated
the structure of a group by imposing some assumptions on the set, for example, de Giovanni and Robinson in \cite{de} proved that
if $G$ is a locally graded group with the set being finite, then $G$ is finite-by-abelian(or $G'$ is finite).
Along the line of \cite{de} and \cite{Herzog}, by imposing different assumptions on the set, more results were obtained in \cite{de Mari1,de Mari2, Rinauro,Zarrin}.
We notice that these works mentioned above studied mainly infinite groups by imposing some assumptions on the set.
Our works are the first in this direction of finite $p$-groups.

\medskip
We use $\exp(G)$, ${\rm cl}(G)$ and $d(G)$ to denote the exponent, the nilpotency class and the minimal cardinality of generating set of $G$, respectively.
For any positive integer $s$,
we define
\begin{center}
$\Omega_{s}(G)=\lg a\in G\ |\ a^{p^{s}}=1\rg$ and $\mho_{s}(G)=\lg a^{p^{s}}\ | \ a\in G\rg$.
\end{center}
Denote by ${\rm C}_{p^n}$ and ${\rm C}_{p^n}^m$  the cyclic group of order $p^n$ and the direct product of $m$ copies  of  ${\rm C}_{p^n}$, respectively. If $A$ and $B$ are subgroups of $G$ with $G=AB$ and $[A,B]=1$, then $G$ is called a central product of $A$ and $B$, denoted by $G=A\ast B$. For a nilpotent group $G$, let
$$G=K_{1}(G)> G'=K_{2}(G)> K_{3}(G)> \cdots> K_{c+1}(G)=1.$$
denote the lower central series of $G$, where $K_{i+1}(G)=[K_{i}(G),G]$ and $c={\rm cl}(G)$.
For other notation and terminology the reader is referred to \cite{BJI}.

\section{Preliminaries}

In this section, we list some known results which be used in the sequel.

A finite group is said to be {\it minimal non-abelian} if it is not abelian but all of its proper subgroups are abelian. The following lemma \ref{minimal non-abelian} is a direct result of \cite[Kapitel III, \S5, Aufgaben 14]{Huppert}.

\begin{lem}\label{minimal non-abelian}
Let $G$ be a minimal non-abelian group. Then $G'$ is a $p$-group. Moreover, if $G$ is a minimal non-abelian $p$-group, then $|G'|=p$.
\end{lem}

\begin{lem}\label{lem=3 grp of max class}{\rm (\cite[Lemma 2.4 ]{ZZLS} )}
Let $G$ be a $3$-group of maximal class. Then the fundamental subgroup $G_{1}$ of $G$ is either abelian or minimal non-abelian.
\end{lem}

\begin{lem}\label{lem=max class}{\rm (\cite[Theorem 9.6(e)]{BJI})}
Let $G$ be a group of maximal class of order $p^n$, where $p>2$ and $n\geqslant p+2$. If $M$ is a maximal subgroup of $G$ and $M\ne G_1$, then $M$ is of maximal class, where $G_{1}=C_{G}(K_{2}(G)/K_{4}(G))$.
\end{lem}

\begin{lem}\label{eryuanyoujiaohuan}{\rm (\cite[Lemma 3.1]{two-generator})}
Let $G$ be a non-abelian  $p$-group with $d(G)=2$ and ${\rm cl}(G)=c$. If $G$ has an abelian subgroup of index $p$, then $Z(G)\cap G'=K_c(G)\cong {\rm C}_p$.
\end{lem}

\begin{lem}\label{lem=G' is a set of commutators}{\rm (\cite[Theorem A]{FAH})}
Let $G$ be a finite $p$-group. If $G'$ can be generated by two elements, then $G'=\{[x,g]~|~g\in G\}$ for a suitable $x\in G$.
\end{lem}

\begin{lem}{\rm (\cite{metabelian})}\label{metabelian formula}
Let $G$ be a metabelian group. Then for each $x,y\in G$, we have
\begin{center}
$[x^{n},y]=\prod\limits_{i=1}^{n}[x,y,(i-1)x]^{n\choose i}$, where $[x,y,(i-1)x]=[x,y,\underbrace{x,x,\ldots,x}\limits_{i-1}]$.
\end{center}
\end{lem}

As a direct consequence of \cite[Theorem 3.1]{metabelian}, we get the following lemma \ref{yajiaohuan}.

\begin{lem}\label{yajiaohuan}
Let $G$ be a two-generator finite metabelian $p$-group. Then $G$ is $p$-abelian if and only if $\exp(G')\leqslant p$ and ${\rm cl}(G)<p$.
\end{lem}

\begin{lem}\label{lem=blackburn}{\rm (\cite[Theorem 4]{Bla})}
Let $G$ be a finite $p$-group. If both $G$ and $G'$ can be generated by two elements, then $G'$ is abelian.
\end{lem}


\begin{lem}\label{lem=Blackburn}{\rm (\cite[Theorem 1.5(ii)]{maximal-class})}
Let $G$ be a finite $p$-group. Then
$$\exp(K_i(G)/K_{i+1}(G))\leqslant \exp(K_{i-1}(G)/K_i(G))~~\mbox{for all}~~i\geqslant 2.$$
\end{lem}

\begin{lem}\label{regular1}{\rm (\cite[Theorem 7.1(b)]{BJI})}
Let $G$ be a finite $p$-group. If ${\rm cl}(G)<p$, then $G$ is regular.
\end{lem}

\begin{lem}\label{regular2}{\rm (\cite[Theorem 7.2(e)]{BJI})}
Let $G$ be a regular $p$-group. Then
$$[x^{p^{k}},y^{p^{n}}]=1\Longleftrightarrow [x,y]^{p^{k+n}}=1.$$
\end{lem}

\begin{lem}\label{lem:=G'=H'G_{3}}{\rm (\cite[Lemma 2.1]{derived})}
Let $G$ be a nilpotent group and $H$ a subgroup of $G$, such that $G'=H'K_{3}(G)$. Then $K_{i}(G)=K_{i}(H)$ for all $i\geqslant 2$.
\end{lem}

\section{Some results of finite ${\mathcal DC}$-groups}
By the definition of ${\mathcal DC}$-groups and Correspondence Theorem, we can get the following useful lemma immediately.
\begin{lem}\label{zishangyichuan}
Let $G$ be a finite ${\mathcal DC}$-group. Then the subgroups and quotient groups of $G$ is also a ${\mathcal DC}$-group.
\end{lem}

\begin{thm}\label{Sylow}
Let $G$ be a finite ${\mathcal DC}$-group. Then

{\rm (1)} $G$ is solvable.

{\rm (2)} $G=P\rtimes A$, where $P\in {\rm Syl}_p(G)$ and $A$ is an abelian $p'$-subgroup of $G$.
\end{thm}

\begin{proof}
(1) Assume $G$ is non-abelian without loss of generality. We use induction on $|G|$. By Lemma~\ref{zishangyichuan}, it suffices to show that $G$ has a non-trivial proper normal subgroup. If $G$ is minimal non-abelian, then $G'$ is a $p$-group by Lemma \ref{minimal non-abelian}. Hence $G$ is solvable.
Assume $G$ is not minimal non-abelian. Then $G$ has a non-abelian
proper subgroup $H$. Obviously, $|H'|=|(H^g)'|$ for any $g \in G$. It follows by $G\in {\mathcal DC}$ that
$H'=(H^g)'$. Since $(H^g)'=H'^g$, $H'=H'^{g}$. Thus $H'\trianglelefteq G$.

\vspace{0.10cm}
(2) Assume $G$ is non-abelian without loss of generality. Let $H$ be a minimal non-abelian subgroup of $G$. Then $H'$ is a $p$-group by Lemma \ref{minimal non-abelian}. Since $G\in {\mathcal DC}$, $G$ is solvable  by (1). Hence $G$ has a $p'$-Hall subgroup $A$. Clearly, $H'\cap A'=1$. Since $G\in {\mathcal DC}$, $A'=1$. For any prime $q\ne p$, let $x\in N_G(Q)$ and $L=\lg Q,x\rg$, where $Q\in {\rm Syl}_q(G)$. Then $L'$ is a $q$-group. Hence $H'\cap L'=1$. Since $G\in {\mathcal DC}$, $L'=1$. That implies $x\in C_G(Q)$. It follows that $N_G(Q)=C_G(Q)$. By the well-known Burnside normal $p$-complement
theorem, $G$ has a normal $q$-complement $C(q)$. Thus $P=\bigcap\limits_{q\ne p}C(q)\trianglelefteq G$. It is easy to see $P$ is a Sylow $p$-subgroup of $G$. The result holds.
\end{proof}

\begin{rem}
A ${\mathcal DC}$-group may be non-supersolvable and may has a non-abelian Sylow subgroup. For example,
let $G=SL(2,3)(\cong {\rm Q}_8\rtimes {\rm C}_3)$. Then $G$ is non-supersolvable and ${\rm Q}_8$ is a non-abelian Sylow $2$-subgroup of $G$. It is easy to get ${\rm Q}_8$ is the unique non-abelian proper subgroup of $G$. Thus $G\in {\mathcal DC}$.
\end{rem}

If $G$ is a finite $p$-group and $G\in {\mathcal DC}$, then
we said that $G$ is a ${\mathcal DC}$ $p$-group. Based on Theorem \ref{Sylow}$(2)$, the study of finite ${\mathcal DC}$-groups can be reduced to that of ${\mathcal DC}$ $p$-groups. The following theorem \ref{lizi} shows that the class of $\mathcal{DC}$ $p$-groups are quite large.

\begin{thm}\label{lizi}
Let $G$ be one of the following finite $p$-groups. Then $G$ is a ${\mathcal DC}$-group.

$(1)$ A group with cyclic derived subgroup.

$(2)$ A two-generator non-abelian group with an abelian maximal subgroup.

$(3)$ A group of maximal class of order $p^{n}$ whose the fundament subgroup $G_{1}$ satisfying $|G_{1}'|=p$, where $n\geqslant p+2$ and $p>2$. In particular, by Lemma {\rm \ref{lem=3 grp of max class}}, groups of maximal class of order $3^{n}$ also are $\mathcal {DC}$-groups, where $n\geqslant 5$.
\end{thm}

\begin{proof}
$(1)$ It is obvious.

\vspace{0.10cm}
$(2)$ We give a proof by induction on ${\rm cl}(G)$. If ${\rm cl}(G)=2$, then, by Lemma~\ref{eryuanyoujiaohuan}, $G'=K_2(G)\cong {\rm C}_p$. The conclusion holds. Assume ${\rm cl}(G)=c>2$. Let $\og=G/K_c(G)$. Then ${\rm cl}(\og)=c-1$.
By induction hypothesis, $\overline{G}\in {\mathcal DC}$. Let $H$ be a  non-abelian subgroup of $G$. We assert that $K_c(G)\leq H'$:
Let $A$ be an abelian maximal subgroup of $G$.  Clearly, $G=AH$ and $H'\leq A$. Hence $H'\trianglelefteq G$. Thus $H'\cap Z(G)\ne 1$. By Lemma~\ref{eryuanyoujiaohuan}, $G'\cap Z(G)=K_c(G)\cong {\rm C}_p$. It follows that $H'\cap Z(G)=K_c(G)$. The assertion holds. Thus  ${\rm DS}(G)$
is also a chain by Correspondence Theorem. Therefore $G\in {\mathcal DC}$.

\vspace{0.10cm}
$(3)$ Let $\og=G/G_{1}'$ and $\overline{G_{1}}=G_{1}/G_{1}'$. Then $d(\overline{G})=2$ and $\overline{G_{1}}$ is an abelian maximal subgroup of $\overline{G}$. It follows by $(2)$ that $\overline{G}\in {\mathcal DC}$.
Let $H$ be a non-abelian subgroup of $G$. By Correspondence Theorem it is enough to show $G_{1}'\leq H'$.
Take $K\leq H$ and $K$ is minimal non-abelian. Clearly, it is enough to show $K'=G_{1}'$.
Since $|K'|=|G_1'|=p$, we only need to show $K'\leq G_1'$. Assume $K'\nleq G_{1}'$. Then $K\nleq G_{1}$ and Hence $G=KG_{1}$. It is easy to get that
\begin{center}
$[K'G_{1}',K]=[K',K]=1$ and $[K'G_{1}', G_{1}]\leq [G',G_{1}]\leq [G_{1},G_{1}]=G_{1}'.$
\end{center}
Thus $K'G_{1}'\unlhd G$. By ${\rm N/C}$ Theorem,
$|G:C_G(K'G_{1}')|\leqslant p$. Let $L$ be a maximal subgroup of $G$ which contained in $C_G(K'G_{1}')$. Then
$ K'G_{1}'\leq Z(L)$. For each maximal subgroup $M\ne G_{1}$, we have $M$ is of maximal class by Lemma \ref{lem=max class}.
Hence $|Z(M)|=p$. Thus $L=G_{1}$. That means $G_{1}\leq C_G(K'G_{1}')$. It follows that $[K'G_{1}',G_{1}]=1$. In particular, $[K',G_{1}]=1$. Thus $[K',G]=[K',KG_1]=1$. That means, $K'\leq Z(G)$. So $K'G_{1}'\leq Z(G)$. This contradicts that $|Z(G)|=p$.
\end{proof}

The following results illustrate that
for a non-abelian ${\mathcal DC}$ $p$-group $G$, $d(G)$ can not be bounded. However, $d(G')$ can be bounded.

\begin{prop}\label{pianxuji}
Let $G$ be a finite non-abelian $p$-group and $A$ a finite $p$-group. Then 

$(1)$ if $A'=1$, then $G\ast A\in {\mathcal DC}$ if and only if $G\in {\mathcal DC}$.

$(2)$ $G\times A\in {\mathcal DC}$ if and only if $G\in {\mathcal DC}$ and $A'=1$.
\end{prop}

\begin{proof}
$(1)$ We only need to prove that ${\rm DS}(G\ast A)={\rm DS}(G)$. Obviously, ${\rm DS}(G)\subseteq {\rm DS}(G\ast A)$.
Take $H\leq G\ast A$. Let $K=\{g\in G\di \exists a\in A ~s.t.~ ga\in H\}$. Then $K\leq G$, and hence $H'=K'$.
Thus ${\rm DS}(G\ast A)\subseteq {\rm DS}(G)$. The result holds.

\vspace{0.10cm}
$(2)$ $(\Leftarrow)$ It is an immediate consequence of $(1)$.

$(\Rightarrow)$ Obviously, $G\in {\mathcal DC}$ by Lemma \ref{zishangyichuan}.
Since $G\times A\in {\mathcal DC}$,  $G'\cap A'=A'$ or $G'\cap A'=G'$. It follows from $G\cap A=1$ and $G'\ne 1$ that $A'=1$.
\end{proof}

In order to prove $d(G')$ has an upper bound, we need the following lemmas.

\begin{lem}\label{Pxingzhi1}
Let $G$ be a finite non-abelian ${\mathcal DC}$ $p$-group. Then

{\rm (1)} $K_{i}(G)/K_{i+1}(G)$ is cyclic for all $2\leqslant i\leqslant {\rm cl}(G)$.

{\rm (2)} there exist $a_i\in K_{i-1}(G),~b_i\in G$ such that $K_i(G)=\lg a_i,b_i\rg'$, where $2\leqslant i\leqslant {\rm cl}(G)$.  In particular,
if $H\leq G$ and $|K_i(G)|\leqslant |H'|$, then $K_i(G)\leq H'$.

{\rm (3)}  if $d(G')=2$, then $G'\cap Z(G)$ is cyclic.

{\rm (4)} if $\exp(G')=p$, then $G'\cap Z(G)=K_{c}(G)\cong {\rm C}_{p}$, where $c={\rm cl}(G)$.
\end{lem}

\begin{proof}
(1) Assume $G$ is a counterexample of minimal order. Then there exists $t$ such that $K_{t}(G)/K_{t+1}(G)$ is non-cyclic. Take a normal subgroup $N$ of $G$ satisfying $K_{t+1}(G)\leq N\leq K_t(G)$ and $K_t(G)/N\cong {\rm C}^{2}_p$. Let $\og=G/N$. Obviously, $\overline{G}\in {\mathcal DC}$ by Lemma \ref{zishangyichuan}. If $N\neq 1$, then $K_{t}(\overline{G})/K_{t+1}(\overline{G})=K_t(G)/N\cong {\rm C}^{2}_p$. This contradicts the minimality of $G$. It follows that $N=1$. Thus $K_{t+1}(G)=1$ and $K_{t}(G)\cong {\rm C}^{2}_p$. Hence $K_{t}(G)\leq Z(G)$.
Therefore there exist $a_{1},a_{2}\in K_{t-1}(G),~b_{1},b_{2}\in G$ such that $K_t(G)=\lg [a_{1},b_{1}]\rg\times \lg [a_{2},b_{2}]\rg$. In particular, $\lg a_{1},b_{1}\rg'\cap \lg a_{2},b_{2}\rg'=1$.
This contradicts $G\in {\mathcal DC}$.

\vspace{0.10cm}
(2) We prove by backward induction on $i$. If $i= {\rm cl}(G)$, then $K_i(G)$ is cyclic by (1). Hence the result holds.
Assume $K_{i+1}(G)=\lg a_{i+1},b_{i+1}\rg'$. Then $K_i(G)=\lg [a_i,b_i], K_{i+1}(G)\rg$ by (1), where $a_i\in K_{i-1}(G),~b_i\in G$.
In particular, $\lg a_i,b_i\rg'\nleq K_{i+1}(G)$.  Since $G\in {\mathcal DC}$,
$K_{i+1}(G)=\lg a_{i+1},b_{i+1}\rg'\leq \lg a_i,b_i\rg'$. Hence $K_i(G)=\lg a_i,b_i\rg'$.

\vspace{0.10cm}
(3) If the result does not hold, then, by Lemma~\ref{lem=G' is a set of commutators}, there exist $x,y,z\in G$ such that $G'\cap Z(G)\ge \lg [x,y],[x,z]\rg\cong {\rm C}_p^2$. It follows that $\lg x,y\rg'=\lg [x,y]\rg \ne \lg [x,z]\rg=\lg x,z\rg'$. This contradicts $G\in {\mathcal DC}$.

\vspace{0.10cm}
(4) Let $N=G'\cap Z(G)$. Then there exists $t$ such that $N\leq K_{t}(G)$ and $N\nleq K_{t+1}(G)$. Since $G\in {\mathcal DC}$, $K_{t}(G)/K_{t+1}(G)$ is cyclic by $(1)$. It follows from $\exp(G')=p$ that $|K_{t}(G)/K_{t+1}(G)|=p$. Hence $K_{t}(G)=NK_{t+1}(G)$. Moreover, $$K_{t+1}(G)=[K_{t}(G),G]=[NK_{t+1}(G),G]=[K_{t+1}(G),G]=K_{t+2}(G).$$
Since $G$ is nilpotent, $K_{t+1}(G)=1$. Hence ${\rm cl}(G)=t$, $N=K_{t}(G)$ and $|K_{t}(G)|=p$.
\end{proof}

\begin{prop}
Let $G$ be a finite non-abelian regular ${\mathcal DC}$ $p$-group, $p$ odd prime. Then $d(G')\leqslant p-2$.
\end{prop}

\begin{proof}
Assume $G$ is a counterexample of minimal order. Then $d(G')\geqslant p-1$.
Let $\overline{G}=G/\Phi(G')$. Then $d(\og')=d(G')$.
 By the minimality of $G$, $\Phi(G')=1$. Hence $G'\cong {\rm C}_{p}^{d(G')}$. So ${\rm cl}(G)=d(G')+1\geqslant p$ by Lemma~\ref{Pxingzhi1}$(1)$.
On the other hand, from Lemma~\ref{Pxingzhi1}$(2)$ and the minimality of $G$ it follows that $d(G)=2$.
Since $G$ is regular and $\mho_{1}(G')=1$, $G$ is $p$-abelian. It follows  by Lemma \ref{yajiaohuan} that ${\rm cl}(G)<p$.
This is a contradiction.
\end{proof}

\begin{lem}\label{daoqunshengchengyuan1}
Let $G$ be a finite non-abelian ${\mathcal DC}$ $p$-group with an abelian maximal subgroup. Then $d(G')\leqslant p-1$.
\end{lem}

\begin{proof}
Let $\overline{G}=G/\Phi(G')$. Then $d(\overline{G}')=d(G')$. Hence we can assume that $\Phi(G')=1$. By Lemma~\ref{Pxingzhi1}$(2)$ assume that $d(G)=2$.
Let $A$ be an abelian maximal subgroup of $G$. Take $a\in A\setminus \Phi(G)$ and $b\in G\setminus A$. Then
 $G=\lg a,b\rg.$ Let $e_{1}=[a,b]$ and $e_{i+1}=[e_{i},b]$, where $i\geqslant1.$  By Lemma~\ref{Pxingzhi1}$(1)$, $K_{i+1}(G)=\lg e_i\rg K_{i+2}(G)$ for $i\geqslant 1$. In particular, $G'=\lg e_1,e_2,\cdots,e_i,\cdots\rg$. Notice that $b^p\in A$. Then $[a,b^p]=1$.
 On the other hand, by Lemma \ref{metabelian formula},  $[a,b^{p}]=\prod\limits_{i=1}^{p}e_i^{p\choose i}=e_{p}$. Thus $e_i=1$ for $i\geqslant p$. It follows that $|G'|\leqslant p^{p-1}$.
 Therefore $d(G')\leqslant p-1$.
\end{proof}

The following Theorem \ref{thm:=main} and \ref{thm=d(G')=p->G' ele-abel} give the proof of Theorem \ref{thm1.2}.

\begin{thm}\label{thm:=main}
Let $G$ be a finite non-abelian ${\mathcal DC}$ $p$-group. Then $|G':\mho_1(G')|\leqslant p^p$. In particular, $d(G')\leqslant p$.
\end{thm}

\begin{proof}
Assume $G$ is a counterexample of minimal order. Then $|G':\mho_1(G')|\geqslant p^{p+1}$. Take $N\unlhd G$, $\mho_1(G')\le N\le G'$ and $|G':N|=p^{p+1}$.  Let $\og=G/N$.
Then $\og$ is also a counterexample. By the minimality of $G$, $N$=1.
It follows that $|G'|=p^{p+1}$ and $\mho_1(G')=1$.

Let $M$ be a maximal subgroup of $G$ such that $|M'|$ has minimal order. Then $M'\unlhd G$. Let $\overline{G}=G/M'$ and $\overline{M}=M/M'$. Then $\overline{M}$ is an abelian
maximal subgroup of $\overline{G}$.
It follows by Lemma~\ref{daoqunshengchengyuan1} that $d(\overline{G}')\leqslant p-1$.
Obviously, $\Phi(\overline{G}')=1$. Thus $|\overline{G}'|\leqslant p^{p-1}$. Since $|G'|=p^{p+1}$, $|M'|\geqslant p^{2}$.
Take $N\leq M'$ with $|N|=p^{2}$ and $N\trianglelefteq G$. Then $|G/C_{G}(N)|\leqslant p$ by ${\rm N/C}$ Theorem. Hence there exists a maximal subgroup $M_{1}$ such that $M_{1}\leq C_{G}(N)$. Moreover, $N\leq Z(M_{1})$. Since $G\in {\mathcal DC}$, it follows by the minimality of $|M'|$ that $M'\leq M_1'$. Hence $N\leq M'_{1}$. Thus $N\leq Z(M_{1})\cap M'_{1}$.  This contradicts Lemma~\ref{Pxingzhi1}$(4)$.
\end{proof}

\begin{thm}\label{thm=d(G')=p->G' ele-abel}
Let $G$ be a finite non-abelian ${\mathcal DC}$ $p$-group. If $d(G')=p$, then $G'\cong {\rm C}_p^p$.
\end{thm}

\begin{proof} Assume $G$ is a counterexample of minimal order. By considering the quotient group of $G$, we can get $|G'|=p^{p+1}$ and $|\Phi(G')|=p$. By Lemma~\ref{Pxingzhi1}$(2)$ assume that $d(G)=2$.
If $p=2$, then $d(G')=2$. It follows from  Lemma \ref{lem=blackburn} that $G'$ is abelian. Moreover, $G'\cong {\rm C}_4\times {\rm C}_2$.
In particular, $|G':\O_1(G')|=2$.  Hence $\exp(K_3(G))=2$.
If $p>2$, then $G'$ is regular since $|\Phi(G')|=p$. Thus $|G':\O_1(G')|=|\mho_1(G')|\leqslant |\Phi(G')|=p$. We still have $\exp(K_3(G))=p$.
Let $\og=G/\Phi(G')$.
It follows by Lemma~\ref{Pxingzhi1}$(1)$ that ${\rm cl}(\og)=p+1$.
Hence $K_{p+2}(G)\leq \Phi(G')$. Thus ${\rm cl}(G)=p+2$ or ${\rm cl}(G)=p+1$.

\vspace{0.20cm}
\f\textbf{Case 1.} ${\rm cl}(G)=p+2$

\vspace{0.20cm}
In this case, we have $K_{p+2}(G)=\Phi(G')$. Since $|G'|=p^{p+1}$ and ${\rm cl}(G)=p+2$, we have $|K_i(G)/K_{i+1}(G)|=p$ for all $2\leqslant i\leqslant p+2$.

Let $M_i=C_G(K_i(G)/K_{i+2}(G)):=\{g\in G\ |\ [K_i(G),g]\leq K_{i+2}(G)\}$, where $2\leqslant i\leqslant p+1$.
Then, by ${\rm N/C}$ Theorem, $M_i$ is a maximal subgroup of $G$. Notice that a finite $p$-group can not be a union of $p$ proper subgroups. Take

\begin{center}
$a\in G\setminus (\bigcup\limits_{i=2}^{p+1}M_i)$,\ \ $b\in M_{p+1}\setminus \Phi(G)$.
\end{center}

By Lemma~\ref{Pxingzhi1}$(2)$ and the minimality of $G$ we have $d(G)=2$. Then $G=\lg a,b\rg$.
Moreover, $K_i(G)=\lg [a,b,(i-2)a], K_{i+1}(G)\rg$, where $2\leqslant i\leqslant p+2$.
Let $H=\lg a^p,b\rg$.  Clearly, $\overline{G}=G/\Phi(G')=G/K_{p+2}(G)$ is matabelian. By Lemma \ref{metabelian formula}, we have

\begin{center}
$[a^p,b]\equiv [a,b]^{p\choose 1}\cdots[a,b,(i-1)a]^{p\choose i}\cdots [a,b,(p-1)a]^{p\choose p}~(\mod K_{p+2}(G)).$
\end{center}

\f Hence $[a^p,b]\equiv [a,b,(p-1)a]~(\mod K_{p+2}(G))$. That means, $[a^p,b]\in K_{p+1}(G)\setminus K_{p+2}(G)$. Since $b\in M_{p+1}=C_G(K_{p+1}(G))$, $[a^p,b,b]=1$. Notice that ${\rm cl}(\lg K_{p+1}(G),a\rg)=2$. We have $[a^p,b,a^p]=[a^p,b,a]^p=1$. Thus $H'=\lg [a^{p},b]\rg$.
Since $\exp(K_3(G))=p$, $|H'|=p$. But $H'\nleq K_{p+2}(G)$.
This contradicts Lemma~\ref{Pxingzhi1}$(2)$.

\vspace{0.20cm}
\f\textbf{Case 2.} ${\rm cl}(G)=p+1$

\vspace{0.20cm}
In this case, we assert that $p\geqslant 3$. If $p=2$, then $G'\cong {\rm C}_4\times {\rm C}_2$. By Lemma~\ref{Pxingzhi1}$(1)$, $K_3(G)\ne \mho_1(G')$. Let $\overline{G}=G/\O_1(G')$.
Then $K_3(G) \leq \O_1(G')$. Hence $\O_1(G')=K_3(G)\mho_1(G')$. Since $|\mho_1(G')|=2$ and $\mho_1(G')\unlhd G$, $\mho_1(G')\leq Z(G)$. By Lemma~\ref{Pxingzhi1}$(3)$, $\O_1(G')\nleq Z(G)$. Thus $K_3(G)\nleq Z(G)$. That means ${\rm cl}(G)=4$. This contradiction ${\rm cl}(G)=3$.

Since ${\rm cl}(G)=p+1$,  by Lemma~\ref{Pxingzhi1}$(1)$ and Lemma~\ref{lem=Blackburn},

\begin{center}
$K_2(G)/K_3(G)\cong {\rm C}_{p^2}$ and $K_i(G)/K_{i+1}(G)\cong {\rm C}_p$ for $3\leqslant i\leqslant p+1$.
\end{center}

\f In particular, $\mho_1(G')\nleq K_3(G)$ and $|K_{p}|=p^{2}$.

Let $M=C_G(K_p(G))$.
Then, by ${\rm N/C}$ Theorem, $M$ is a maximal subgroup of $G$.  Take $a\in G\setminus M$ and $b\in M\setminus \Phi(G)$. Then $G=\lg a,b\rg$.
Let $H=\lg a^p,b\rg$ and $\og=G/K_p(G)$. Then ${\rm cl}(\og)<p$ and hence $\og$ is regular by Lemma \ref{regular1}.
Since $[\bar{a}, \bar{b}]^p\ne \bar{1}$, $[\bar{a}^p,\bar{b}]\neq 1$ by Lemma \ref{regular2}. Hence $[a^p,b]\notin K_p(G)$. On the other hand, it is easy to see that $[a^p,b]\in \mho_1(G')K_p(G)$. Thus $[a^p,b]=xy$, where $x\in \mho_1(G')$, $y\in K_p(G)$ and $x\ne 1$. Since $|\mho_1(G')|=p$ and $\mho_1(G')\unlhd G$, $\mho_1(G')\leq Z(G)$. Thus $x\in Z(G)$. Since $b\in C_G(K_p(G))$ and $y\in K_p(G)$, $[b,y]=1$. Notice that ${\rm cl}(\lg K_{p}(G),a\rg)=2$. We have $[y,a^p]=[y,a]^p=1$. So $H'=\lg xy\rg$. That means $|H'|=p$. Hence $|H'|\leqslant |K_{p}(G)|$. It follows by Lemma~\ref{Pxingzhi1}$(2)$ that $H'\leq K_p(G)$. Thus $x\in K_p(G)$. Since
$p>2$, $x\in K_3(G)$. This forced $\mho_1(G')\leq K_3(G)$. This contradicts $\mho_1(G')\nleq K_3(G)$.
\end{proof}

\begin{rem}
By Theorem {\rm \ref{thm=d(G')=p->G' ele-abel}}, we see that if  $d(G')=p$ for a $\mathcal{DC}$-group $G$, then $\exp(G')=p$. However, when $d(G')<p$,  $\exp(G')$  can be as large as possible. For example, Let $G$ be a $p$-group of maximal class with an abelian maximal subgroup and $|G|= p^{k(p-1)+2}$. Then $G\in \mathcal{DC}$, $d(G')=p-1$ and $\exp(G')=p^k$.
For the structure of $p$-group of maximal class with an abelian maximal subgroup, the readers is suggested to see {\rm\cite{Bla}}.

\end{rem}

{\it The proof of Theorem {\rm \ref{DC2-group}}}:
$(\Rightarrow)$ Assume $G$ is non-abelian without loss of generality. It follows by Theorem \ref{thm:=main} that $d(G')\leqslant 2$. If $d(G')=2$, then $G'\cong {\rm C}_2^2$ by Theorem \ref{thm=d(G')=p->G' ele-abel}. Hence ${\rm cl}(G)\leqslant 3$. By Lemma~\ref{Pxingzhi1}$(3)$, we have ${\rm cl}(G)=3$.

\vspace{0.2cm}
$(\Leftarrow)$ If $G'$ is cyclic, then $G\in {\mathcal DC}$ by Theorem \ref{lizi}$(1)$. If $G'\cong {\rm C}_{2}^{2}$ and ${\rm cl}(G)=3$, then $|K_{3}(G)|=2$. Let $H$ be a non-abelian subgroup of $G$. If $H'\ne K_3(G)$, then $G'=H'K_3(G)$. It follows by Lemma \ref{lem:=G'=H'G_{3}} that $H'=G'$. It implies that ${\rm DS}(G)=\{ 1,K_3(G),G'\}$. So $G\in {\mathcal DC}$. \qed

\begin{cor}\label{lem=s.n. conditon of 2 gen DC 2 grp}
Let $G$ be a two-generator finite non-abelian $2$-group. Then $G\in {\mathcal DC}$ if and only if $G'$ is cyclic or $G'\cong {\rm C}^{2}_{2}$.
In particular, the groups with $d(G)=2$ and $G'\cong {\rm C}_{2}^{2}$
were classified up to isomorphism in {\rm \cite[Theorem 4.6]{minimal non-abelian}}.
\end{cor}

Obviously, a finite ${\mathcal DC}$ $2$-group is metabelian. It is not difficult to get, by Lemma~\ref{lem=blackburn}, \ref{Pxingzhi1}$(2)$ and Theorem \ref{thm:=main},~\ref{thm=d(G')=p->G' ele-abel}, that a finite non-abelian ${\mathcal DC}$ $3$-group is also metabelian. However, the following example show that there exists finite non-abelian ${\mathcal DC}$ $p$-groups are non-metabelian for $p\geqslant 5$.

\begin{exam}
If a group $G$ is one of the following groups, then $G$ is a ${\mathcal DC}$ $p$-group with $G'$ is non-abelian.

\medskip
$(1)$ $G=\lg a,x; a_{i}\ |\ x^{p}=a^{p}=a^{p}_{i}=1,[x,a]=a_{1},[a_{j},a]=a_{j+1},[a_{1},a_{2}]=[x,a_{1}]
=[a_{2},x]=[a_{3},x]=a_{5}\rg$, where $1\leqslant i\leqslant 5$, $1\leqslant j\leqslant 4$ and $p\geqslant 7$, other commutators are 1.

\medskip
$(2)$ $G=\lg a_{1},a_{2};a_{3},a_{4},a_{5}\ |\ a_{1}^{5^2}=a_{3}^{5^2}=a_{2}^5=a_{4}^5=a_{5}^5=1,
[a_{2},a_{1}]=a_{3},[a_{3},a_{2}]=a_{4},[a_{4},a_{2}]=a_{5},[a_{5},a_{2}]=a_{1}^5,
[a_{1},a_{4}]=[a_{3},a_{4}]=[a_{1},a_{5}]=a_{3}^5,[a_{1},a_{3}]=[a_{3},a_{5}]=
[a_{4},a_{5}]=1\rg.$
\end{exam}

\begin{proof} By a simple check we get that $G'$ is non-abelian for the groups $(1)$ and $(2)$, and they have the following same properties.

\vspace{0.1cm}
$(1)$ $|G|=p^{7}$, $Z(G)$ is cyclic and $d(G)=2$.

\vspace{0.1cm}
$(2)$ $G$ has a maximal subgroup $M_{0}$ with $|M_{0}'|=p$.

\vspace{0.1cm}
$(3)$ $Z(M)$ is cyclic for each maximal subgroup $M\neq M_{0}$.

Based on these observation above, we get that $G\in {\mathcal DC}$ by a similar argument as that of Theorem \ref{lizi}$(3)$.
\end{proof}

At last, we propose the following

\vspace{0.25cm}
\textbf{Problem.} Study the derived length ${\rm dl}(G)$ for a finite ${\mathcal DC}$ $p$-group $G$. Note that if $G\in {\mathcal DC}$ and $p\leqslant 3$, then ${\rm dl}(G)\leqslant 2$.


\begin{thebibliography}{99}
\bibitem{minimal non-abelian}
An, L.J., Hu, R.F., Zhang, Q.H.: Finite $p$-groups with a minimal non-abelian subgroup of index $p$ ${\rm (IV)}$. J. Algebra Appl. {\bf 14}, (2015). \url{https://doi.org/10.1142/S0219498815500206}

\bibitem{BJI}
Berkovich, Y.: Groups of Prime Power Order,
Vol.1.  Walter de Gruyter, Berlin and New York (2008)

\bibitem{BJIII}
Berkovich, Y.,  Janko, Z.:  Groups of Prime Power Order,
Vol.3.  Walter de Gruyter, Berlin and New York (2011)

\bibitem{Bla}
Blackburn, N.: On prime-power groups in which the derived group has two generators. Proc. Cambr. Phil.
Soc. {\bf 53}, 19--27 (1957)

\bibitem{maximal-class}
Blackburn, N.: On a special class of $p$-groups. Acta Math. {\bf {100}}, 45--92 (1958)


\bibitem{metabelian}
Brisley, W., Macdonald, I.D.: Two classes of metabelian groups. Math. Z.  {\bf {112}}, 5--12 (1969)


\bibitem{de}
De Giovanni, F., Robinson, D.J.S.: Groups with finitely many derived subgroups. J. Lond. Math. Soc. {\rm (2)}. $\bf 71$, 658--668 (2005)

\bibitem{de Mari1}
De Mari, F., De Giovanni, F.: Groups with finitely many derived subgroups of non-normal subgroups.  Arch. Math.
{\rm(}Basel{\rm)}. $\bf 86$, 310--316 (2006)


\bibitem{de Mari2}
De Mari, F.: Groups with finiteness conditions on the lower central series of non-normal subgroups. Arch. Math. {\rm(}Basel{\rm)}. $\bf 109$, 105--115 (2017)

\bibitem{FAH}
Fern\'{a}ndez-Alcober, G.A., De Las Heras, I.: Commutators in finite $p$-groups with $2$-generator derived subgroup. Israel J. Math. {\bf 232}, 109--124 (2019)


\bibitem{Herzog}
Herzog, M., Longobardi, P., Maj, M.:  On the number of commutators in groups.
Contemp. Math. {\bf 402}, 181--192 (2006)

\bibitem{Huppert}
Huppert, B.:  Endliche Gruppen I. Springer-Verlag, Berlin, Heidelberg and New York (1967)


\bibitem{Rinauro}
Rinauro, S.:  Groups with finiteness conditions on the lower central series of subgroups. Algebra Colloq. $\bf 20$, 663--670 (2013)

\bibitem{derived}
Schneider, C.: Groups of prime-power order with a small second derived quotient. J. Algebra  {\bf {266}}, 539--551 (2003)


\bibitem{two-generator}
Xu, M.Y.,  An, L.J., Zhang, Q.H.: Finite $p$-groups all of whose non-abelian proper subgroups are generated by two elements. J. Algebra  {\bf {319}}, 3603--3620 (2008)

\bibitem{Zarrin}
Zarrin, M.: On groups with finitely many derived subgroups.  J. Algebra Appl.
{\bf 13}, (2014). \url{https://doi.org/10.1142/S0219498814500455}

\bibitem{ZZLS}
Zhang, Q.H., Zhao, L.B., Li, M.M., Shen, Y.Q.: Finite $p$-groups all
of whose subgroups of index $p^3$ are abelian. Commun. Math. Stat. {\bf 3}, 69--162 (2015)

\end{thebibliography}
\end{document}